\documentclass{amsart}
\usepackage{graphicx}
\usepackage{amsthm}
\usepackage{amssymb}
\usepackage[bottom]{footmisc}

\newtheorem{definition}{Definition}
\newtheorem{lem}{Lemma}
\newtheorem{prop}{Proposition}
\newtheorem{theorem}{Theorem}
\newtheorem{cor}{Corollary}

\newcommand\blfootnote[1]{%
  \begingroup
  \renewcommand\thefootnote{}\footnote{#1}%
  \addtocounter{footnote}{-1}%
  \endgroup
}

\begin{document} 
\blfootnote{\textup{2000} \textit{Primary 52A21, 52A20; Secondary	 46B70}:
code}
\keywords{Banach Mazur distance}
\title{A remark on geodesics in the Banach Mazur distance}
\author{Alvaro Arias}
\author{Vladimir Kovalchuk}
\address{University of Denver}
\email{alvaro.arias@du.edu, vladimir.kovalchuk@du.edu}

\begin{abstract}
We show that there are uncountably many geodesics between any two non-isometric $n$-dimensional normed spaces.  We construct two explicit geodesics that can be used to describe all the points of the other geodesics.
\end{abstract}

\maketitle
\section{Introduction}
The (multiplicative) Banach-Mazur distance between the $n$-dimensional normed spaces $E$ and $F$ is given by 
\[
d(E,F)=\inf\left\{ \|T\|\|T^{-1}\|:T:E\to F\text{ is an isomorphism}\right\}.
\]
This quantity was introduced by Pe\l czy\'nski and it measures how close $E$ and $F$ are isomorphic to each other.  One can easily check that the infimum is attained, that  $d(E,F)=1$ iff $E$ and $F$ are isometric, and that $d(E,F)\leq d(E,H)d(H,F)$.  It follows that $\log d(E,F)$ is a distance in the Banach Mazur compactum $BM_n$, the set of isometry classes of $n$-dimensional normed spaces, that turns out to be compact with this metric.

The Banach Mazur distance has a geometric interpretation.  Suppose that $E=(\mathbb{R}^n,B_E)$ and $F=(\mathbb{R}^n,B_F)$, where $B_E$ and $B_F$ are the unit balls of $E$ and $F$.  Then one can easily check that $d(E,F)$ is the smallest number $L\geq 1$ such that there exists an invertible map $T:F\to E$ such that $B_E\subset T(B_F)\subset L B_E$.  If $T:F\to E$ attains the distance and if we replace $F$ by its isometric version $(\mathbb{R}^n,T(B_F))$, we assume without loss of generality that 
\begin{equation}\label{eq:1}
B_E\subset B_F \subset d(E,F) B_E.
\end{equation}

This means that when the normed spaces are put in this canonical position, the distance is attained by the identity map. 

\medbreak

Suppose that $(M,\rho)$ is a metric space.  The length of a path $\gamma:[a,b]\to (M,\rho)$ is defined by
\[
L(\gamma)=\sup_{a=t_0\leq t_1\leq\cdots\leq t_n=b}\sum_{i=1}^n \rho(\gamma(t_i),\gamma(t_{i-1})),
\]
where the supremum runs over all possible partitions.  By the triangle inequality, we have that 
$\rho(\gamma(a),\gamma(b))\leq L(\gamma)$.  A {\sl geodesic} between $x\in M$ and $y\in M$ is a path $\gamma$ that starts at $x$, ends at $y$, and that has length equal to $\rho(x,y)$.  

\begin{definition} A metric space $(M,\rho)$ is a {\sl geodesic space} if every two points in $M$ are joined by a geodesic.
\end{definition}

For the (multiplicative) distance in the Banach Mazur compactum $BM_n$, the relevant concept 
is the following:

\begin{definition} The $n$-dimensional normed space $H$ is an intermediate space between $E$ and $F$ if $d(E,F)=d(E,H)d(H,F).$
\end{definition}

We easily see that 

\begin{lem} \label{characterization}
A path $\gamma:[a,b]\to BM_n$ is a geodesic from $E$ to $F$ iff $\gamma(a)=E$, $\gamma(b)=F$, and for every $a=t_0<t_1<\cdots <t_n=b$, 
$d(E,F)=\prod_{i=1}^n d(\gamma(t_i),\gamma(t_{i-1}))$.
\end{lem}

Classical spaces provide examples of geodesics.  It follows from Holder's inequality that the paths $\{\ell_p^n:1\leq p\leq 2\}$ and $\{\ell_q^n:2\leq q\leq \infty\}$ are geodesics in $BM_n$.  Indeed, if $1\leq r<s\leq \infty$, $B_{\ell_r^n}\subset B_{\ell_s^n}\subset n^{\frac{1}{r}-\frac{1}{s}}B_{\ell_r^n}$.  When $2\leq q<\infty$ we have 
\[
B_{\ell_2^n}\subset B_{\ell_q^n}\subset n^{\frac{1}{2}-\frac{1}{q}}B_{\ell_2^n}
\quad\quad \text{ and } \quad\quad 
B_{\ell_q^n}\subset B_{\ell_\infty^n}\subset n^{\frac{1}{q}} B_{\ell_q^n}.
\]
This implies that $d(\ell_2^n,\ell_q^n)\leq  n^{\frac{1}{2}-\frac{1}{q}}$ and that $d(\ell_q^n,\ell_\infty^n)\leq n^{\frac{1}{q}}$.
Since 
\[
\sqrt{n}=d(\ell_2^n,\ell_\infty^n)\leq d(\ell_2^n,\ell_q^n)d(\ell_q^n,\ell_\infty^n)\leq n^{\frac{1}{2}-\frac{1}{q}} n^{\frac{1}{q}}   =\sqrt{n},
\]
we conclude that $d(\ell_2^n,\ell_q^n) =  n^{\frac{1}{2}-\frac{1}{q}}$ and  $d(\ell_q^n,\ell_\infty^n) = n^{\frac{1}{q}}$.  If $2<q_1<q_2<\infty$, then $B_{\ell_{q_1}}^n\subset B_{\ell_{q_2}}^n\subset
n^{\frac{1}{q_1}-\frac{1}{q_2}} B_{\ell_{q_1}}^n$ and this implies that $d(\ell_{q_1}^n,\ell_{q_2}^n)\leq n^{\frac{1}{q_1}-\frac{1}{q_2}}$.  Since
\[
n^{\frac{1}{2}-\frac{1}{q_2}}=d(\ell_2^n,\ell_{q_2}^n)\leq d(\ell_2^n,\ell_{q_1}^n)d(\ell_{q_1}^n,\ell_{q_2}^n)
\leq n^{\frac{1}{2}-\frac{1}{q_1}}n^{\frac{1}{q_1}-\frac{1}{q_2}}=n^{\frac{1}{2}-\frac{1}{q_2}},
\]
we conclude that $d(\ell_{q_1}^n,\ell_{q_2}^n) =  n^{\frac{1}{q_1}-\frac{1}{q_2}}$.  And from this, it follows that $\{\ell_q^n:2\leq q\leq \infty\}$ is a geodesic.  The case $\{\ell_p^n:1\leq p\leq 2\}$ is similar. 

Notice that the path $\{\ell_p^n:1\leq p\leq \infty\}$ is not a geodesic because $d(\ell_1^n,\ell_\infty^n)<d(\ell_1^n,\ell_2^n)d(\ell_2^n,\ell_\infty^m)=\sqrt{n}\sqrt{n}=n$ (see [4], for the case of $n=3$ see [5]).

Interpolation spaces can be used to obtain geodesics and the previous examples illustrate this for $\ell_p^n$ spaces.  The technique works in more general settings, including in operator spaces (see [1]).  Our results can be interpreted using the real interpolation $K$-method and $J$-method.  However, we prefer to use geometric language.

\section{Characterization of Intermediate Spaces}

In this section we show that there are many intermediate spaces between any two non-isomorphic $n$-dimensional normed spaces.  We first identify the ``extreme'' intermediate spaces and then we use them to determine all of them and to describe geodesics.

We fix some notation.   $E$ and $F$ are two non-isomorphic $n$-dimensional spaces that satisfy 
$B_E\subset B_F \subset d(E,F) B_E$. For $0<\lambda<1$, let $E_\lambda=(\mathbb{R}^n,B_\lambda)$ and $F_\lambda=(\mathbb{R}^n,C_\lambda)$ be the normed spaces with unit balls defined by  
\begin{equation}\label{eq:2}
B_\lambda=(d(E,F)^\lambda B_E)\cap B_F \quad \quad\text{and}\quad\quad C_\lambda=\text{Conv}\left[  B_E\bigcup \frac{1}{d(E,F)^{1-\lambda}} B_F    \right],
\end{equation}
where Conv$(S)$ is the convex hull of the set $S$.  

\begin{lem}\label{inclusion}
For $0<\lambda<1$, $B_E\subset C_\lambda\subset B_\lambda\subset B_F$
\end{lem}

\begin{proof} Let $0<\lambda<1$ and let $d=d(E,F)$.  $B_E\subset C_\lambda$ and $B_\lambda\subset B_F$ follow from the definition of $C_\lambda$ and $B_\lambda$.  We need to show that $C_\lambda\subset B_\lambda$.  By convexity of $B_\lambda$, we only need to check that $B_E\subset B_\lambda$ and that $B_F\subset d(E,F)^{1-\lambda}B_\lambda$.  Since $d^\lambda\geq 1$, $B_E\subset d^\lambda B_E$; and since $B_E\subset B_F$ (from (\ref{eq:1})), we conclude that  $B_E\subset B_\lambda$.  To check the other inclusion, notice that $d^{1-\lambda}B_\lambda=(dB_E)\cap(d^{1-\lambda}B_F)$.  Then $B_F\subset dB_E$ from (\ref{eq:1}) and since $d^{1-\lambda}\geq1$, $B_F\subset d^{1-\lambda}B_F$.
\end{proof}

Using the notation of (\ref{eq:2}) we have

\begin{prop}\label{extreme_spaces}
For $0<\lambda<1$, the spaces $E_\lambda$ and $F_\lambda$ are intermediate spaces between $E$ and $F$.  More precisely, $d(E,E_\lambda)=d(E,F_\lambda)=d(E,F)^\lambda$ and $d(E_\lambda,F)=d(F_\lambda,F)=d(E,F)^{1-\lambda}$.
\end{prop}
\begin{proof} Let $0<\lambda<1$ and $d=d(E,F)$.
We start with $E_\lambda$.  We claim that $B_E\subset B_\lambda\subset d^\lambda B_E$.  The first inclusion follows from Lemma \ref{inclusion} and the second follows from the definition of $B_\lambda$.  Notice that this implies that $d(E,E_\lambda)\leq d^\lambda$.  Now we claim that 
$B_\lambda\subset B_F\subset d^{1-\lambda}B_\lambda$. The first inclusion follows from Lemma \ref{inclusion}.  The definition of $C_\lambda$ implies that $B_F\subset d^{1-\lambda}C_\lambda$.  Then by Lemma \ref{inclusion} again, $B_F\subset d^{1-\lambda}C_\lambda\subset d^{1-\lambda}B_\lambda$.   Notice that this implies that $d(E_\lambda,F)\leq d^{1-\lambda}$.  Since
$d=d(E,F)\leq d(E,E_\lambda)d(E_\lambda,F)\leq d^\lambda d^{1-\lambda}=d$, we conclude that $d(E,E_\lambda)=d^\lambda$ and that $d(E_\lambda,F)=d^{1-\lambda}$.

The proof of $F_\lambda$ is similar.  We claim that $B_E\subset C_\lambda\subset d^\lambda B_E$. 
The first inclusion follows from Lemma \ref{inclusion} and the second one follows from $C_\lambda\subset B_\lambda$ and $B_\lambda\subset d^\lambda B_E$, which we proved in the previous parragraph.  We also claim that $C_\lambda\subset B_F\subset d^{1-\lambda}C_\lambda$.  The first inclusion follows from Lemma \ref{inclusion} and the second one follows from the definition of $C_\lambda$.  Following the argument of the previous paragraph, we conclude that $d(E,F_\lambda)=d^\lambda$ and $d(F_\lambda,F)=d^{1-\lambda}$.
\end{proof}

\begin{cor}\label{geodesics}
The sets $\{B_\lambda:0\leq\lambda\leq1\}$ and $\{C_\lambda:0\leq\lambda\leq1\}$ are geodesics from $F$ to $E$.
\end{cor}
\begin{proof} 
Let $d=d(E,F)$ and $0<\lambda_1<\lambda_2<1$.  We claim that $B_{\lambda_1}\subset B_{\lambda_2}\subset d^{\lambda_2-\lambda_1}B_{\lambda_1}$.  The first inclusion follows from the definition of $B_\lambda$.  To check the second inclusion, notice that $d^{\lambda_2-\lambda_1}B_{\lambda_1}=(d^{\lambda_2}B_E)\cap (d^{\lambda_2-\lambda_1}B_F)$ and  $B_{\lambda_2}=(d^{\lambda_2}B_E)\cap B_F$.  Since $d^{\lambda_2-\lambda_1}\geq1$, $B_F\subset  d^{\lambda_2-\lambda_1}B_F$ and this implies that $B_{\lambda_2}\subset d^{\lambda_2-\lambda_1}B_{\lambda_1}$.  

From $B_{\lambda_1}\subset B_{\lambda_2}\subset d^{\lambda_2-\lambda_1}B_{\lambda_1}$ it follows that $d(E_{\lambda_1},E_{\lambda_2})\leq d^{\lambda_2 - \lambda_1}$.  Since $d(E,E_{\lambda_1})=d^{\lambda_1}$ and 
$
d^{\lambda_2} =d(E,E_{\lambda_2})\leq d(E,E_{\lambda_1})d(E_{\lambda_1},E_{\lambda_2})\leq d^{\lambda_1} d^{\lambda_2-\lambda_1}=d^{\lambda_2},
$
we conclude that $d(E_{\lambda_1},E_{\lambda_2})=d^{\lambda_2-\lambda_1}$.

A similar argument shows that $C_{\lambda_1}\subset C_{\lambda_2}\subset d^{\lambda_2-\lambda_1}C_{\lambda_1}$ and this implies that $d(F_{\lambda_1},F_{\lambda_2})=d^{\lambda_2-\lambda_1}$.
\end{proof}

\noindent {\bf Remark. } One can check that the norms of $E_\lambda$ and $F_\lambda$ are given the real $K$-method and $J$-method ([3], pp. 96 - 105).
$$\|x\|_{E_\lambda}= K(x,d(E,F)^\lambda,E,F) \quad\text{and}\quad \|x\|_{F_\lambda}=J\left(x,\frac{1}{d(E,F)^{1-\lambda}},F,E\right)$$
\smallbreak

We now use the intermediate spaces of Proposition \ref{extreme_spaces} to describe all other intermediate spaces.  If $X$ is an intermediate space between $E$ and $F$, then $d(E,X)=d(E,F)^\lambda$ for some $\lambda\in[0,1]$.  This number and the unit balls of the previous Proposition determine the intermediate spaces.

\begin{theorem}
Suppose that $E, F, X$ are $n$-dimensional normed spaces.  Then $X$ is an intermediate space between $E$ and $F$ iff there exist $\lambda\in[0,1]$ and isometric copies of $E,F,X$ in 
$\mathbb{R}^n$ such that $d(E,X)=d(E,F)^\lambda$, $B_E\subset B_F\subset d(E,F) B_F$ and
$
C_\lambda\subset B_X\subset B_\lambda. 
$
\end{theorem}

\begin{proof}
Suppose that $X$ is an intermediate space between $E$ and $F$.  Then there exists $\lambda\in[0,1]$ such that $d(E,X)=d(E,F)^\lambda$ and $d(X,F)=d(E,F)^{1-\lambda}$.  Find $T:X\to F$ and $S:E\to X$  such that $\|T\|=1$, $\|T^{-1}\|=d(E,F)^{1-\lambda}$, $\|S\|=1$ and 
$\|S^{-1}\|=d(E,F)^\lambda$.  Then $T(B_X)\subset B_F\subset d(E,F)^{1-\lambda}T(B_X)$ and 
$S(B_E)\subset B_X\subset d(E,F)^\lambda S(B_E).$

Replacing $E$ and $X$ by their isometries $(\mathbb{R}^n,T(S(B_E)))$ and $(\mathbb{R}^n,T(B_X))$ we get 
\[ B_E\subset B_X\subset d(E,F)^\lambda B_E\quad\text{ and }\quad
B_X\subset B_F\subset d(E,F)^{1-\lambda} B_X.
\]
Combining these inclusions we get that $B_E\subset B_F\subset d(E,F) B_E$.  Moreover, we clearly have $B_X\subset \left(d(E,F)^\lambda B_E     \right)\cap B_F$ and
$\text{Conv}\left[ B_E\bigcup \frac{1}{d(E,F)^{1-\lambda}} B_F \right]\subset B_X$.

\medbreak

On the other hand, suppose that there are isometric versions of $E$, $F$, and $X$ that satisfy
$B_E\subset B_F\subset d(E,F) B_E$ and $C_\lambda\subset B_X\subset B_\lambda$.  
From the proof of Proposition \ref{extreme_spaces}, 
we have that $B_E\subset C_\lambda\subset d(E,F)^\lambda B_E$ and
$B_E\subset B_\lambda\subset d(E,F)^\lambda B_E$.  Therefore we have that 
$B_E\subset C_\lambda\subset B_X\subset B_\lambda\subset d(E,F)^\lambda B_E$, that implies that $d(E,X)\leq d(E,F)^\lambda$.

Similarly, $C_\lambda\subset B_F\subset d(E,F)^{1-\lambda}C_\lambda$ and $B_\lambda\subset B_F\subset d(E,F)^{1-\lambda}B_\lambda$.  Then 
\[
B_X\subset B_\lambda\subset B_F\subset d(E,F)^{1-\lambda}C_\lambda\subset d(E,F)^{1-\lambda} B_X,
\]
and this implies that $d(X,F)\leq d(E,F)^{1-\lambda}$.  Since $d(E,F)\leq d(E,X)d(X,F)\leq d(E,F)^\lambda d(E,F)^{1-\lambda} =d(E,F)$, we conclude that $X$ is an intermediate space between $E$ and $F$.
\end{proof}

We now refine Lemma \ref{inclusion}. 

\begin{lem} \label{separation}
Suppose $E$ and $F$ are non-isometric $n$-dimensional normed spaces satisfying $B_E\subset B_F \subset d(E,F)B_E$.  Then $C_{\lambda} \subsetneq B_{\lambda}$ for all $\lambda \in (0,1)$.
\end{lem}

\begin{proof}

Let $\mathcal{C}:= \partial B_E \cap \partial B_F$. Then $\mathcal{C}$ is closed, non-empty (we need to have contact points between the spheres to attain the distance) and $\mathcal{C} \neq \partial B_F$ (or $d(E,F)=1$). Then $\partial B_f \setminus \mathcal{C} $ is open in the relative topology of $\partial B_F$ and $\partial B_F \setminus \mathcal{C}$ is not closed (because $\partial B_F$ is connected). Find $x\in\mathcal{C}$ and $x_n \in \partial B_F\setminus \mathcal{C}$ such that $x_n \rightarrow x$.

Since $x \in \partial B_E$, $x \in d(E,F)^{\lambda}B_E^{\circ}$ (the interior) and we can find $n$ large enough so that $x_n \in d(E,F)^\lambda B_E^{\circ}$.  It is clear that this $x_n$ belongs to $B_\lambda$ and we will show that it does not belong to $C_\lambda$.  Find $f: \mathbb{R}^n \to \mathbb{R}$ linear, separating $x_n$ from $B_F^\circ$.  Then we can assume that $f(x_n)=1$ and that for all $y \in B_F, f(y) \leq 1$. If $x_n$ belonged to $C_{\lambda}$ we would write it as $x_n = \alpha y_1 + (1-\alpha)y_2$ for some $\alpha \in [0,1], y_1 \in B_E, y_2 \in \frac{1}{d(E,F)^{1-\lambda}}B_F$.  Note that $f(y_1)\leq 1$ and $f(y_2)<1$. Then  $1=f(x_n)=\alpha f(y_1)+ (1-\alpha)f(y_2)$ implies that $\alpha=1$ and $x_n=y_1$ which is not possible.  Therefore, $x_n\not\in C_\lambda$.
\end{proof}

\section{Main Result}

In this section we show there are uncountably many different geodesics between two non-isometric $n$-dimensional normed spaces.   We will start recalling some standard definitions.

\begin{definition}
 Set $(x,y):=\{ \alpha x + (1 - \alpha) y : \alpha \in (0,1) \}$.  A face $F$ of a convex set $K$ is a subset of $K$ satisfying the following: if $z \in F$, $x,y \in K$ and $ z \in (x,y)$ then $x,y \in F $.  A face $F$ is {\sl exposed} if there exists a separating hyperplane $H$ such that $F= H \cap K$.  If $F$ has a non-empty relative interior in $F\cap H$, $F$ is an $(n-1)$-dimensional face of $K$. The set of $(n-1)$-dimensional faces is denoted by $\mathrm{F^{n-1}}(K)$. 
\end{definition}

\begin{lem} \label{perfect faces}
The unit ball of a finite dimensional normed space $E$ has at most countably many $(n-1)$-dimensional faces.   
\end{lem}

\begin{proof}
Let $\{F_\alpha:\alpha\in I\}$ be the set of (n-1)-dimensional faces of $B_E\subset\mathbb{R}^n$, the unit ball of $E$.  Their interiors with respect to the relative topology of the sphere $S_E$, the boundary of $B_E$, form a disjoint family of non-empty open sets in $S_E$.  Since $S_E$ is second countable, $I$ is at most countable.
\end{proof}

The next proposition states if $E$ and $F$ are isometric, then 
$\mathrm{F^{n-1}(B_E)}$ and 
$\mathrm{F^{n-1}(B_F)}$ are equal up to affine maps.  This provides a criterion to show that two normed spaces are not isometric.  We are to exhibit an $n-1$ dimensional face that is not in the other.  Notice that there are uncountably many non-isometric convex bodies in $\mathbb{R}^{n-1}$, if $n\geq 3$.  Since (after an affine map) any convex body $K$ in $\mathbb{R}^{n-1}$ can be a face of an $n$-dimensional normed space, there are a lot of options.

We need the following result, that is easy to prove:

\begin{lem} \label{affine copy}
Let $E$ and $F$ be isometric $n$-dimensional normed spaces and let $Q \in \mathrm{F^{n-1}(B_E)}$.  Then there exists $Q' \in \mathrm{F^{n-1}(B_E)}$ that is an affine copy of $F$. 
\end{lem}

We state and show our main result in two parts, first one deals with $BM_n$ for $n \geq 3 $, the other with $BM_2$.

\begin{theorem}\label{main theorem}
Let $n\geq3$ and let $E$ and $F$ be two non-isometric $n$-dimensional normed spaces satisfying $B_E\subset B_F\subset d(E,F) B_E$.  Then for each $\lambda \in (0,1)$ the set $\{ X \in BM_n : C_{\lambda} \subset B_X\subset B_{\lambda} \}$ contains uncountably many non-isometric spaces.  
\end{theorem}

\begin{proof}
 By Lemma \ref{separation} 
there exists $x \in B_{\lambda}\setminus C_{\lambda}$ and since $C_\lambda$ is closed, we can assume that $x\in B_{\lambda}^{\circ}\setminus C_\lambda$.  Find $\epsilon>0$ such that $B_{\epsilon} (x) \subset B_{\lambda}^{\circ} \setminus C_{\lambda}$.  Then find a linear function $f:\mathbb{R}^n\to\mathbb{R}$ seperating $\{x\}$ from $C_\lambda$.  Assume that $f(x)=1$ and that for all $y\in C_\lambda$, $f(y)<1$.  Let $H=\{z\in\mathbb{R}^n:f(z)=1\}$ be the hyperplane induced by $f$ at $x$. 
  By Lemma \ref{perfect faces}  
the collections $\mathrm{F}^{n-1}B_{\lambda},$ $\mathrm{F}^{n-1}C_{\lambda}$ are countable. Since there are uncountably many non-isometric $n-1$-dimensional Banach spaces, choose a unit ball $K$ in $\mathbb{R}^{n-1}$ that is not affinely isometric to a ball in either collection. Find an affine copy of $K$ inside $H \cap B_\epsilon(x)$ and call it $K'$. Define the $n$-dimensional normed space $X_K$ to have unit ball equal 
  $ B_X: =B_{X_K}:=$Conv$\left(C_\lambda \cup K' \cup -K'\right).$
Then $C_\lambda\subset B_{X_K}\subset B_\lambda$ and $B_{X_K}$ has an $n-1$ dimensional face isometric to $K$. So by Lemma \ref{affine copy} 
it is not isometric to $C_{\lambda}$ nor $B_{\lambda}$. Since the construction produces uncountably many non-isometric spaces the claim follows.
\end{proof}
The above construction fails for $n=2$ since all 1-dimensional normed spaces are isometric. We tackle the $n=2$ case in the next section.   We need some elementary results to show that there are uncountably many geodesics between $E$ and $F$.  

\begin{lem}\label{intermediate}
Suppose that $X$ is an intermediate space between $E$ and $F$, that $Y$ is an intermediate space between $E$ and $X$ and that $Z$ is an intermediate space between $X$ and $F$.  Then $X$ is an intermediate space between $Y$ and $Z$.
\end{lem}
\begin{proof}
This follows from the triangle inequality:
$d(E,F)\leq d(E,Y)d(Y,Z)d(Z,F)\leq d(E,Y)d(Y,X)d(X,Z)d(Z,F)=d(E,X)d(X,F)=d(E,F).$  Then the inequalities are equalities and $d(Y,Z)=d(Y,X)d(X,Z)$.
\end{proof}
\begin{cor} \label{intermediate geodesic}
If $X$ is an intermediate space between $E$ and $F$, then there exists a geodesic from $E$ to $F$ containing $X$. 
\end{cor}
\begin{proof}
Use Corollary \ref{geodesics} to construct geodesics from $E$ to $X$, from $X$ and $F$ and put them together.  To see that the joined path is a geodesic from $F$ to $E$ we use Lemma \ref{intermediate} to show it satisfies the condition of Lemma \ref{characterization}.  If the partition contains $X$, it follows from the fact that $X$ is an intermediate space and that both pieces are geodesics.  And if the partition does not include $X$, we use Lemma \ref{intermediate}, adding $X$ to the partition.
\end{proof}

Combining Theorem \ref{main theorem} with the previous Corollary we get:

\begin{cor}
There are uncountably many geodesics between any two non-isomorphic $n$-dimensional normed spaces for $n\geq3$.  
\end{cor}

\section{Dimension 2}

 In what follows $K$ is a compact, convex and symmetric body of $\mathbb{R}^2$. By Lemma \ref{perfect faces} 
there are at most countably many perfect $1$-faces \textit{i.e.} line segments 
$\mathrm{F}^{1}K$. Consider the set of triangles formed by joining the endpoints of line segments $[p,q]$ in $F^1K$ to the origin, $\bigtriangleup 0pq$ .  Call this set 
$\mathcal{S}_K:= \{ \bigtriangleup 0pq : [p,q] \in \mathrm{F}^1K  \}$ and consider all possible ratios of areas of these triangles  $\mathcal{A}_K:= \{  \frac{\mu(T_1)}{\mu(T_2)} : T_1,T_2 \in S_K  \}$, where $\mu$ is the usual Lebesgue area measure. Since $\mu(\phi(T))=$det$(\phi)(\mu(T))$ we arrive at the following countable isometric invariant:

\begin{lem}\label{invariant}
Let $\phi : (\mathbb{R}^2, B_E) \to (\mathbb{R}^2, B_F)$ be an invertible linear map with $\phi(B_E)= B_F$. Then $\mathcal{A}_{B_E}=\mathcal{A}_{B_F}$.
\end{lem}
Now we prove Theorem \ref{main theorem}, 
for $n=2$.  The idea is to find balls $B_q$ satisfying $C_\lambda\subset B_q\subset B_\lambda$, with two contiguous faces $[p_1,q]$ and $[q,p_2]$ such that 
$\frac{\mu(\bigtriangleup 0p_1q)}{ \mu(\bigtriangleup 0p_2q)  }$ 
does not belong to $\mathcal{A}_{B_E}$ or $\mathcal{A}_{B_F}$.

\begin{theorem}\label{n=2}
Suppose that $n=2$ and that $E$ and $F$ are two non-isometric normed spaces satisfying $B_E\subset B_F\subset d(E,F) B_E$.  Then for each $\lambda \in (0,1)$ the set $\{ X \in BM_n : C_{\lambda} \subset B_X\subset B_{\lambda} \}$ contains uncountably many non-isometric spaces.  
\end{theorem}

\begin{proof}
 By Lemma \ref{separation} 
there exists $x \in B_{\lambda}\setminus C_{\lambda}$ and since $C_\lambda$ is closed, we can assume that $x\in B_{\lambda}^{\circ}\setminus C_\lambda$.  Find $\epsilon>0$ so that $\overline{B_{\epsilon} (x)} \subset B_{\lambda}^{\circ} \setminus C_{\lambda}$ and find a linear function $f:\mathbb{R}^n\to\mathbb{R}$ separating $\overline{B(x,\epsilon)}$ from $C_{\lambda}$. Then $H=\{z\in\mathbb{R}^2:f(z)=f(x)\}$ is the hyperplane (line in this case) induced by $f$ at $x$; and $H\cap \overline{B(x,\epsilon)}$ is a segment that we denote $[p_1,p_2]$.

As a first step, let $B=\text{conv}(C_\lambda\cup[p_1,p_2]\cup[-p_1,-p_2])$.  Notice that $C_\lambda\subset B\subset B_\lambda$ and that the segment $[p_1,p_2]$ is a face of 
$B$.
For each $q\in\overline{B(x,\epsilon)}$ with $f(q)>f(x)$ (i.e., $q$ is inside $\overline{B(x,\epsilon)}$ but outside the triangle $\bigtriangleup 0p_1p_2$) define
$$B_q=\text{conv}(C_\lambda\cup[p_1,p_2]\cup[-p_1,-p_2])\cup\{q,-q\}).$$
We still have that $C_\lambda\subset B_q\subset B_\lambda$ and it is easy to check that, for $i=1,2$, the segment $[p_i,q]$ is a face of $B_q$ iff the line going through $p_i$ and $q$ does not intersect $C_\lambda$.  Since $f$ separates $\overline{B(x,\epsilon)}$ and $C_\lambda$ there are many points $q$ with this property.  
In fact, it is easy to see that there exists $0<\delta<\epsilon$ small enough so that whenever $q\in\overline{B(x,\delta)}$ and $f(q)>f(x)$, then the segments $[p_1,q]$ and $[q,p_2]$ are faces of $B_q$.

Find two such points $q_1$ and $q_2$ that satisfy $f(q_1)=f(q_2)>f(x)$.  Notice that the segment $[q_1,q_2]$ is parallel to the segment $[p_1,p_2]$ and that for any $q\in[q_1,q_2]$, the segments  $[p_1,q]$ and $[q,p_2]$ are faces of $B_q$.  
Since $[q_1,q_2]$ is connected and since the values of 
$\frac{\mu(\bigtriangleup 0p_1q)}{ \mu(\bigtriangleup 0p_2q)  }$ 
when $q=q_1$ and $q=q_2$ are clearly different, the set of points 
$\left\{  
\frac{\mu(\bigtriangleup 0p_1q)}{ \mu(\bigtriangleup 0p_2q)  }
:q\in[q_1,q_2]\right\}$
is uncountable.  This allows us to choose $q$'s such that $(\mathbb{R}^2, B_q)$ is not isometric to 
$(\mathbb{R}^2, C_\lambda)$ or to $(\mathbb{R}^2, B_\lambda)$.  Moreover, we can easily choose uncountably many non-isometric $(\mathbb{R}^2, B_q)$'s and the result follows.
\end{proof}

This Theorem combined with Corollary \ref{intermediate geodesic} imply that there are uncountably many geodesics between $E$ and $F$, which completes the case $n\geq2$.
The proof of Theorem \ref{n=2} can be adapted to prove Theorem \ref{main theorem} 
for $n \geq 3$.
   
\smallbreak

\bibliographystyle{amsplain}

\end{document}